\documentclass[11pt]{amsart}
\pdfoutput=1
\usepackage[utf8]{inputenc}
\usepackage{amsmath, amssymb}
\usepackage[english]{babel}
\usepackage{graphicx}
\usepackage{amsthm}
\usepackage{array}
\usepackage{amsthm}
\usepackage{mathtools}
\usepackage{thmtools}
\usepackage[margin=.85 in]{geometry}
\usepackage{float}
\allowdisplaybreaks

\usepackage[dvipsnames]{xcolor}
\usepackage{hyperref}
\definecolor{candyapplered}{rgb}{1.0, 0.03, 0.0}
\definecolor{mediumblue}{rgb}{0.0, 0.0, 0.8}
\hypersetup{
pdfauthor={Tarik Aougab and Jonah Gaster},
pdftitle={k-systems on the torus},
bookmarksnumbered,
colorlinks=true,
linkcolor=mediumblue,
citecolor=candyapplered,
urlcolor=blue}

\usepackage{secdot}
\usepackage{caption}
\newlength{\bibitemsep}
\newlength{\bibparskip}
\let\oldthebibliography\thebibliography
\renewcommand\thebibliography[1]{%
  \oldthebibliography{#1}%
  \setlength{\parskip}{\bibitemsep}%
  \setlength{\itemsep}{\bibparskip}%
}

\usepackage{caption,subcaption}
\usepackage[ocgcolorlinks]{ocgx2}

\declaretheorem[numberwithin=section]{theorem}

\declaretheorem[sibling=theorem]{lemma}
\declaretheorem[sibling=theorem, style=remark]{remark}
\declaretheorem[sibling=theorem]{proposition}
\declaretheorem[sibling=theorem]{conjecture}
\declaretheorem[sibling=theorem]{corollary}

\numberwithin{equation}{section}

\mathtoolsset{showonlyrefs=true}

\begin{document}
\title[Origamis associated to minimally-intersecting filling pairs]
{Origamis associated to minimally-intersecting filling pairs}
\author[Aougab, Menasco, Nieland]{Tarik Aougab, William W. Menasco, and Mark Nieland}

\address{Department of Mathematics, Haverford College}
\email{taougab@haverford.edu}

\address{Department of Mathematics, University at Buffalo}
\email{menasco@buffalo.edu}

\address{School of Mathematical Sciences, Rochester Institute of Technology}
\email{mansma@rit.edu}


\begin{abstract}
Let $S_{g}$ denote the closed orientable surface of genus $g$.  In joint work with Huang, the first author constructed exponentially-many (in $g$) mapping class group orbits of pairs of simple closed curves whose complement is a single topological disk. Using different techniques, we improve on this result by constructing factorially-many (again in $g$) such orbits. These new orbits are chosen so that the absolute value of the algebraic intersection number is equal to the geometric intersection number, implying that each pair naturally gives rise to an origami. We collect some rudimentary experimental data on the corresponding $SL(2, \mathbb{Z})$-orbits and suggest further study and conjectures. 

\end{abstract}

\maketitle

\section{Introduction}

Let $S_{g}$ denote the closed orientable surface of genus $g$. A finite collection of curves $\Gamma= \left\{\gamma_{1}, ... \gamma_{n} \right\}$ in pairwise-minimal position on $S_{g}$ is said to \textit{fill} if the complement $S_{g} \setminus \Gamma$ is a disjoint union of topological disks. When each $\gamma_{i}$ is simple and $n= 2$, we call $\Gamma$ a \textit{filling pair}. 

Thurston used filling pairs to construct pseudo-Anosov mapping classes by composing Dehn twists in the appropriate way around the curves in the pair \cite{TS}. The dilatation of these pseudo-Anosovs directly relates to the geometric intersection number of the pair. Penner later generalized Thurston's construction to pairs of filling multicurves \cite{P}. Any pseudo-Anosov obtainable by such a construction is known as a \textit{Penner-Thurston map}.

Filling multicurve pairs with small intersection number yield maps with relatively small dilatation amongst all Penner-Thurston maps. Motivated by this and other connections to the mapping class group-- which will henceforth be denoted by $\mbox{Mod}(S_{g})$-- in joint work with Huang \cite{AH}, the first author constructed exponentially-many (in $g$) $ \mbox{Mod}(S_g)$-orbits of filling pairs whose geometric intersection number is the minimum possible among all filling pairs on $S_{g}$.

In this paper, using a new construction, we build factorially-many minimally-intersecting filling pairs, all of which satisfy the additional condition that the absolute value of the algebraic intersection is equal to the geometric intersection number. 

\begin{theorem} \label{main} For odd $g \geq 3$, there exist at least $(g-2)!$ distinct $ Mod^{\pm}(S_g) $-orbits of filling pairs $\left\{\gamma_{1}, \gamma_{2} \right\}$ so that the absolute value of the algebraic intersection number is $2g-1$. For even $g >2$, there are at least $(g-5)\cdot (g-3)!$ such orbits. 
\end{theorem}

\subsection{Square-tiled surfaces}

A \textit{square-tiling} of a surface $S= S_{g}$ is a decomposition of $S$ into quadrilaterals with at least four regions around each vertex. If each region is identified with a square whose area is scaled such that the total area is $1$, this induces a unit-area singular flat metric on $S$. The horizontal foliation on each individual square can be extended to a measured singular foliation $\mathcal{F}$ on $S$. A fundamental theorem of Hubbard and Masur then associates a unique quadratic differential to this horizontal foliation and, therefore, to the square-tiling \cite{HM}. In this way, one can interpret a square-tiling of $S$ as a point in the space of unit-area quadratic differentials. If $\mathcal{F}$ is orientable, the square-tiling supports the structure of a \textit{translation surface}, and the associated quadratic differential will be the square of an abelian differential. In this case, we refer to the surface equipped with its square-tiling as a \textit{square-tiled surface}. 

An abelian differential on a surface of genus $ g>1 $ must have $2g-2$ zeros, counted with multiplicity.  The moduli space of abelian differentials is naturally stratified by the degrees of these zeros:  given $m_{1},..., m_{n} \in \mathbb{N}$ satisfying $m_{1}+...+m_{n}= 2g-2$, the \textit{stratum} $\mathcal{H}(m_{1},..., m_{n})$ of the moduli space consists of those abelian differentials which have $n$ zeros of degrees $m_{1},..., m_{n}$. Each stratum admits a volume element -- the \textit{Masur-Veech volume}-- and an invariant $SL(2,\mathbb{R})$-action.  Masur and Veech proved that each stratum has finite volume (\cite{M}, \cite{V1}, \cite{V2}). 

To compute the volume of a stratum of abelian differentials, Eskin-Okounkov (\cite{EO}) use the following approach suggested by Zorich (\cite{Zo}): each stratum can be locally modeled on the relative cohomology $H^{1}(S, P; \mathbb{C})$, where $P$ is the finite set of singularities. The square-tiled surfaces in $ S $ are analogous to ``integer points,'' as their period coordinates lie in $\mathbb{Z} \oplus i \mathbb{Z}$. The asymptotics for the number of square-tiled surfaces constructed from at most $N$ squares in a given stratum can then be used to compute the volume.

For this reason, counting square-tiled surfaces has deep connections to Teichm{\"u}ller dynamics and to the study of flat surfaces. In the literature, a square-tiled surface is also called an \textit{origami} (alternatively defined as a Riemann surface obtained as a finite cover of the square torus branched over a single point). The $SL(2,\mathbb{R})$-action on any stratum restricts to an $SL(2,\mathbb{Z})$-action on the origamis in that stratum, and the study of the corresponding $SL(2,\mathbb{Z})$-orbits is a highly active area of research (\cite{HS} \cite{S}, \cite{HL}, \cite{Mc}). 

A filling pair $\lbrace \alpha, \beta \rbrace$ on $S$ in minimal position naturally gives rise to a decomposition of $S$ into quadilaterals.  We can interpret $\alpha \cup \beta$ as a graph on $S$ whose vertices are elements of $ \alpha \cap \beta $ and whose edges are components of the symmetric difference $ \alpha \bigtriangleup \beta $; the dual of this graph will be the $1$-skeleton of a quadilateral decomposition. If the absolute value of the algebraic intersection number equals the geometric intersection number ($|\hat{\iota}(\alpha, \beta)|=\iota(\alpha, \beta)$), the decomposition yields an origami. In the language of origamis and flat geometry, when $g$ is odd, \autoref{main} can be reframed as: 

\vspace{2 mm} 

\textit{In the minimal stratum $\mathcal{H}(2g-2)$, there exist at least $(g-2)!$ square-tiled surfaces with $2g-1$ squares whose vertical and horizontal cylinder decompositions each consist of one cylinder.} 

\vspace{2 mm}

Since these origamis have the simplest possible horizontal and vertical cylinder decompositions and the minimal number of squares possible, We will refer to such a surface as a \textit{minimal origami}. Using representation-theoretic tools and following the techniques of Zagier (\cite{Za}), Delecroix, Goujard, Zograf, and Zorich compute the exact number of square-tiled surfaces in $\mathcal{H}(2g-2)$ that have $2g-1$ squares and one maximal horizontal cylinder (see Equation (2.21) in \cite{DGZZ}). By contrast, both the vertical and horizontal cylinder decompositions of the origamis from \autoref{main} consist of a single cylinder, and they arise from a completely constructive recipe. 

We advertise the constructive aspect of \autoref{main} as our key contribution---the recipe can be easily implemented on a computer to yield a large repository of origamis in $\mathcal{H}(2g-2)$ with the simplest cylinder decompositions and the smallest number of squares. During the 2018 Summer@ICERM program, a rudimentary version of the recipe was used to identify several distinct $SL(2, \mathbb{Z})$-orbits of these simple origamis; we summarize these findings in Section $5$. 

We conjecture that there should be many distinct $SL(2,\mathbb{Z})$-orbits as genus grows: 

\begin{conjecture} \label{many orbits} Let $N_{g}$ denote the number of distinct $SL(2, \mathbb{Z})$-orbits of origamis in $\mathcal{H}(2g-2)$ that have $2g-1$ squares and whose horizontal and vertical cylinder decompositions each consist of one cylinder. Then $N_{g}$ grows super-exponentially in $g$. 
\end{conjecture}

\subsection{Outline}  In section $2$, we recall some preliminaries for both origamis and curves on surfaces.  In section $3$, we explain how to encode an origami with a minimally-intersecting filling pair as a permutation. In section $4$, we outline the main construction. Finally, in Section $5$, we summarize some results obtained in low genera after computationally implementing the main construction and investigating the orbits under the $SL(2, \mathbb{Z})$-action. 

\subsection{Acknowledgements.} The first author thanks Jayadev Athreya and Samuel Taylor for many helpful conversations. The first author was supported by NSF Grant DMS 1939936 while this work was completed. The authors thank Zichen Cui, Ajeet Gary, Paige Helms, Ionnis Iakovoglu, Tasha Kim, and J.T. Rustad for their work during the 2018 Summer@ICERM program. The authors thank the Institute for Computational and Experimental Research in Mathematics (ICERM) for their hosting of the Summer@ICERM program during which the experimental data in Section $5$ was obtained. Finally, the authors thank Luke Jeffreys for extremely helpful discussions and for noticing an error in the original version of the even genus construction.

\section{preliminaries}

Let $S= S_{g}$ be oriented. If $\alpha, \alpha'$ are simple closed curves on $S$ which are isotopic to one another, we write $\alpha \sim \alpha'$. Then, given essential simple closed curves $\alpha, \beta$ on $S$, the \textit{geometric intersection number}, $\iota(\alpha, \beta)$, is defined by 

\[ \iota(\alpha, \beta) = \min( |\alpha' \cap \beta'| : \alpha' \sim \alpha, \beta' \sim \beta). \] 

The curves $\alpha$ and $\beta$ are said to be in \textit{minimal position} if $|\alpha \cap \beta|= \iota(\alpha, \beta)$. 

Supposing that $\alpha, \beta$ are essential simple closed curves, equip each with a choice of orientation. Then, for $p \in \alpha \cap \beta$, define its \textit{index} $\epsilon(p) \in \left\{1,-1 \right\}$ to be $1$ if the orientation of $p$ from $\alpha$ to $\beta$ given by the right hand rule agrees with the orientation of $S$, and $-1$ otherwise. Then the \textit{algebraic intersection number}, $\hat{\iota}(\alpha, \beta)$, is defined by 

\[ \hat{\iota}(\alpha, \beta) = \sum_{p \in \alpha \cap \beta} \epsilon(p). \]

Note that $\hat{\iota}(\alpha, \beta)= -\hat{\iota}(\beta, \alpha)$, and that the value of $\hat{\iota}(\alpha, \beta)$ does not change when we replace either $\alpha$ (or $\beta$) with another curve in its isotopy class. It follows that the absolute value of $\hat{\iota}$ is well-defined over pairs of isotopy classes $\left\{[\alpha], [\beta] \right\}$. 

The \textit{mapping class group} of $S$, denoted $\mbox{Mod}(S)$, is the group of isotopy classes of orientation-preserving self-homeomorphisms of $S$. The \textit{extended mapping class group} of $S$, denoted $\mbox{Mod}^{\pm}(S)$, is the group of isotopy classes of self-homeomorphisms of $S$, including the orientation-reversing ones. Given two collections of isotopy classes $\Gamma= \left\{\gamma_{1},..., \gamma_{n} \right\}$, $\Omega= \left\{\omega_{1},..., \omega_{n} \right\}$, we say that $\Gamma$ and $\Omega$ are in the same $\mbox{Mod}(S)$ (resp. $\mbox{Mod}^{\pm}(S)$) orbit if there exists some $g \in \mbox{Mod}(S)$ (resp. in $\mbox{Mod}^{\pm}(S)$) so that $g \cdot \gamma_{i} = \omega_{i}$ for $i=1,.., n$. 

We say that a pair of minimal position curves $\left\{\alpha, \beta \right\}$ is a \textit{filling pair} if their complement is a disjoint union of disks. Given a filling pair $\left\{\alpha, \beta \right\}$, let $R_{1},..., R_{n}$ be the complementary disk regions. Since $\alpha \cup \beta$ is a $4$-valent graph on $S$, the graph has twice as many edges as vertices. Using Euler's formula, it follows that 
\[ 2- 2g= \chi(S_{g}) = \iota(\alpha, \beta) - 2\iota(\alpha, \beta) + n \]
\[ \Rightarrow \iota(\alpha, \beta) \geq 2g-1. \]

Thus, when a filling pair $\left\{\alpha, \beta \right\}$ satisfies $\iota(\alpha, \beta) = 2g-1$, we call it a \textit{minimally-intersecting filling pair}. In \cite{AH}, the first author and Huang directly construct exponentially-many $\mbox{Mod}^{\pm}(S)$-orbits of minimally intersecting filling pairs. However, the construction \textit{never} produced filling pairs satisfying $|\hat{\iota}(\alpha, \beta)|= 2g-1$. Thus, the filling pairs constructed below in Section $4$ are all distinct from those constructed in \cite{AH}. 

 An \textit{origami} is a Riemann surface obtained as a finite cover of the torus branched over a single point. The union of the core curves of each cylinder in the horizontal cylinder decomposition of an origami is called the \textit{horizontal multicurve}, denoted $\alpha$; likewise, the union of the vertical core curves is the \textit{vertical multicurve} and is denoted $\beta$.  Since the number of squares of the origami is equal to $ \iota(\alpha, \beta)$, the horizontal and vertical curves of a minimal origami form a minimally-intersecting filling pair.  If $ \lbrace \alpha, \beta \rbrace $ is a minimally-intersecting filling pair with $ \iota(\alpha, \beta)=|\hat{\iota}(\alpha, \beta)|$ (so that the underlying surface is a minimal origami), we refer to it as a \textit{minimal origami pair}.  If we orient $\alpha$ and label the squares from $1$ to $N$ (the total number of squares), we obtain an associated \textit{horizontal permutation} in the following manner:  each cycle of the permutation corresponds to a component of $ \alpha $ and its entries are given by the sequence of squares one obtains by traversing that component in the direction of orientation.  Similarly, we obtain a \textit{vertical permutation} associated with $ \beta $. A relabeling of the squares will produce horizontal and vertical permutations that are simultaneously conjugate to the first pair of permutations.

\section{Filling pairs as permutations} 

Let $\left\{\alpha, \beta \right\}$ be a minimal origami pair.
Then we can represent the surface as a horizontal row of $2g-1$ squares with $ \alpha $ and $ \beta $ as its horizontal and vertical curves.  The horizontal mid-segments of each square will correspond to $\alpha$ (the left- and right-most vertical edges are identified so that these segments form a closed curve). The intersection pattern of $\alpha$ and $\beta$ provides gluing instructions for the remaining edges of the squares, ensuring that each top edge is identified to one bottom edge and that the vertical mid-segments form a single simple closed curve which corresponds to $ \beta $. 

If we number the squares in the row from $1$ to $2g-1$ in a cyclic left-to-right order and orient $ \alpha $ to the right, then its associated horizontal permutation is simply $\sigma_{2g-1}=(1,2, \ldots, 2g-1) $.  We will denote by $\tau_{\alpha, \beta}$ the vertical permutation obtained by orienting $\beta$ so that each vertical mid-segment points upward.  We will see below that this permutation encodes all the relevant combinatorial information about the surface and its filling pair.  Since $ \beta $ is a single curve, $\tau_{\alpha, \beta}$ will necessarily be a $(2g-1)$-cycle.

\autoref{Fig:G3 example} illustrates a genus-$3$ example that is configured to be consistent with our construction in \autoref{main construction}.  At left is an origami in $\mathcal{H}(5)$  with the squares labeled in cyclic left-to-right order (starting with ``$3$'').  The numbering of the upper and lower edges corresponds to their gluing. At right is the filling pair drawn on $S_3$; the labeling of the curve segments corresponds to the labeling of the upper/lower edges. The permutation $\tau_{\alpha, \beta}$ is the one that maps the label of each upper edge to the label of its square.

\vspace{2 mm}

\begin{figure}[h]
\centering
\includegraphics[width=.8\linewidth]{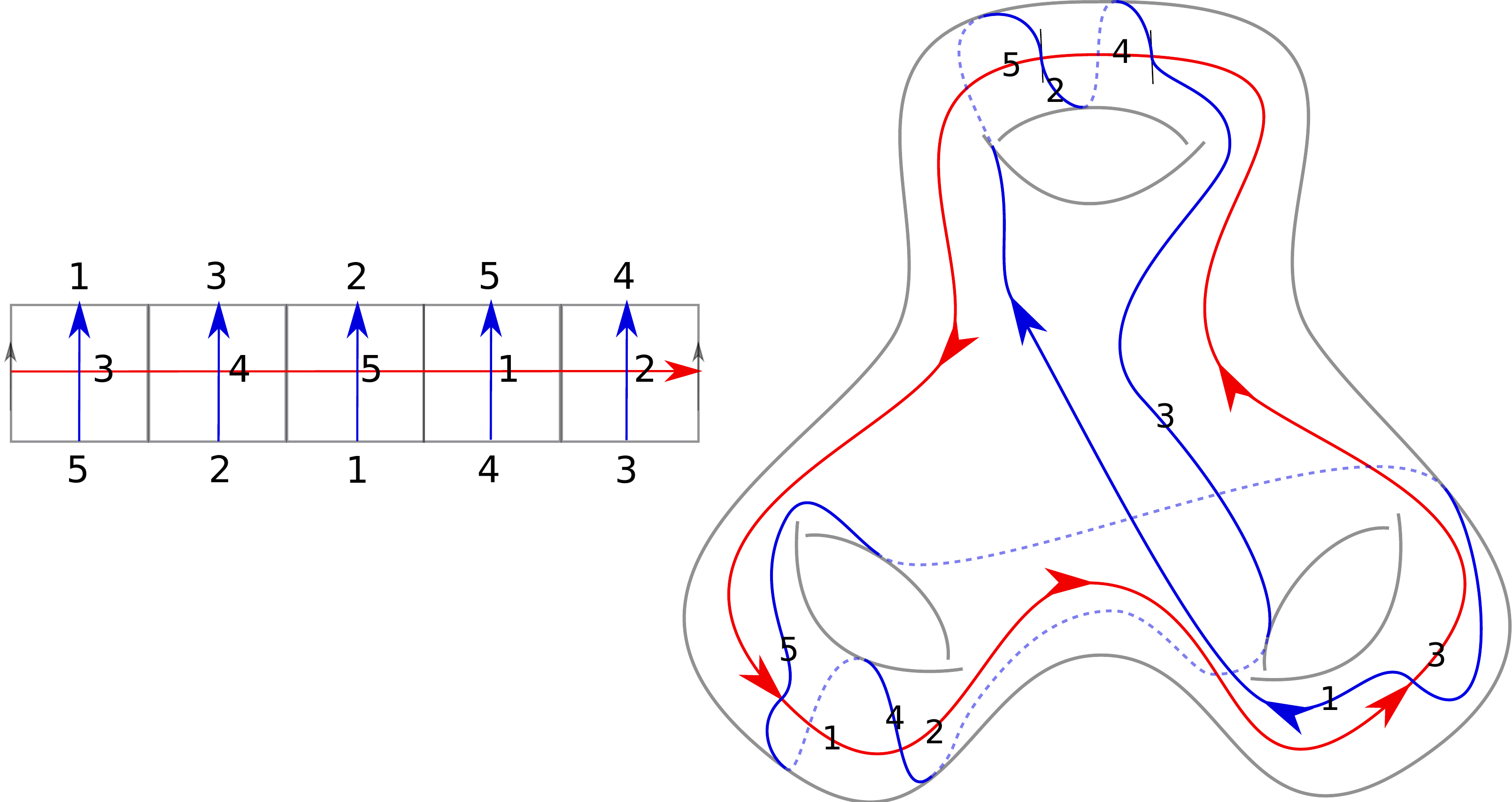}
\caption{\ The left illustrates a filling pair and the associated origami.  The right illustrates its geometric realization on $S_3$. For this example,  $\tau_{\alpha , \beta} = (1,3,5,4,2)$.  Please compare with the $S_7$-example in \autoref{g=7 example}.}
\label{Fig:G3 example}
\end{figure} 

\vspace{2 mm}

\begin{remark}[\textbf{A Note about Left and Right}]
In what follows, we adopt the convention of reading multiplication in the symmetric group from left to right:  if $ \sigma, \tau $ are permutations, then the permutation $ \sigma \tau $ is the one which performs $\sigma$ first and then $\tau$.  This choice affects certain other algebraic expressions that will be relevant to our arguments, namely conjugates and commutators.

When we refer to the operation ``conjugation by $ \tau $," we mean that, if a permutation is expressed in cycle notation, then the conjugate of that permutation is obtained by replacing each entry in that expression with its image under $\tau$.  Accordingly, ``the conjugate of $ \sigma $ by $ \tau $" is written in the form $ \tau^{-1} \sigma \tau $, rather than the form $ \tau \sigma \tau^{-1}$.

Finally, we define the commutator of $ \tau $ and $ \sigma $ by the formula  $ [\tau, \sigma]= \tau \sigma \tau^{-1} \sigma^{-1} $. 
\end{remark}

We can trace around the boundary of a regular neighborhood of $\alpha \cup \beta$ as follows. Each square of $S$ contains one element of $ \alpha \cap \beta$, and $\alpha \cup \beta$ subdivides the square into $4$ quadrants. We then choose a starting intersection point and, beginning in the upper-right quadrant, we travel upwards along $\beta$ to the next intersection (which is contained in the square whose lower edge is glued to the upper edge of the starting square).  Because every intersection has the same index, we arrive in its lower-right quadrant; we then travel along $\alpha$ to the right until we arrive at the next intersection point 
(in the lower-left quadrant of the square).  Next, we travel downwards along $\beta$ (opposite its orientation) to the upper-left quadrant of the next square.  Finally, we travel to the left along $\alpha$, arriving again in the upper-right quadrant of the next square. Repeating this pattern---up along $\beta$, right along $\alpha$, down along $\beta$, and then left along $\alpha$, never crossing $\alpha \cup \beta$---will trace around the boundary of the regular neighborhood.  Since $S \smallsetminus (\alpha \cup \beta)$ is a single disk, there is only one boundary component, so we will travel along both sides of every edge of $\alpha \cup \beta$ before arriving back at the starting place, after $2g-1$ iterations. 

This process can be described with permutations: if we travel along $\beta$ in the direction of orientation beginning at the square labeled $ x $, then our destination will be the square labeled $\tau_{\alpha, \beta}(x)$; if we travel in the direction opposite the orientation, we arrive at $\tau_{\alpha, \beta}^{-1}(x)$.  Similarly, the destination along $ \alpha $ is given by $\sigma_{g}$ or $ \sigma_g^{-1}$ (depending on the direction of travel).  Thus, the boundary components of $\alpha \cup \beta$ correspond to the cyclic factors of $ [\tau_{\alpha, \beta}, \sigma_g ]$.  A regular neighborhood of $ \alpha \cup \beta $ has one boundary component precisely when $ [\tau_{\alpha, \beta}, \sigma_g ] $ is a $(2g-1)$-cycle.

These observations show that if $\tau_{\alpha, \beta}$ is obtained as the vertical permutation of a minimal origami, it must have the following properties:

\begin{enumerate}
\item $\tau_{\alpha, \beta}$ is a $(2g-1)$-cycle 
\item $[\tau_{\alpha, \beta}, \sigma_{g}]$ is a $(2g-1)$-cycle. 
\end{enumerate}

The following result establishes that these properties completely characterize such permutations.

\begin{lemma} \label{characterizing perm} Let $\tau$ be a permutation satisfying the two properties listed above. Then there corresponds a filling pair $\left\{\alpha, \beta \right\}$ on $S_{g}$ satisfying 
\[ \iota(\alpha, \beta) = |\hat{\iota}(\alpha, \beta)|= 2g-1\]
such that $\tau$ is its vertical permutation. Furthermore, two such filling pairs are in the same $\mbox{Mod}^{\pm}(S_{g})$-orbit if and only if their vertical permutations are conjugate to one another by a permutation that commutes with $\sigma_{g}$. 
\end{lemma}

\begin{proof} Starting with a horizontal row of $2g-1$ squares, we can use $\tau$ to specify a gluing pattern for the top and bottom edges, as described above. The fact that $\tau$ is a $(2g-1)$-cycle will guarantee that the vertical mid-segments will join up under this gluing to yield a single simple closed curve which we refer to as $\beta_{\tau}$. Gluing the left- and right-most edges together turns the concatenation of the horizontal mid-segments into a second simple closed curve, denoted by $\alpha_{\tau}$.  It's clear that, if we number the squares from left to right in consecutive order, the horizontal permutation of this origami is $ \sigma_g $ and the vertical permutation is $ \tau $; therefore, the second property of $\tau$ implies that the boundary of a regular neighborhood of $\alpha_{\tau} \cup \beta_{\tau}$ has a single connected component. 

Now, suppose that $\left\{\alpha_{1}, \beta_{1} \right\}$ and $\left\{\alpha_{2}, \beta_{2} \right\}$ are two filling pairs which satisfy the required properties for geometric and algebraic intersection numbers and suppose there is a self-homeomorphism $ f $ of $ S_g $ mapping $\left\{\alpha_{1}, \beta_{1} \right\}$ to $\left\{\alpha_{2}, \beta_{2} \right\}$.  Since $\left\{\alpha_{1}, \beta_{1} \right\}$ is a minimal origami pair, there is a square-tiling $ R_1 $ of the surface which consists of a horizontal strip of $ 2g-1 $ squares whose horizontal curve is $\alpha_{1}$ and whose vertical curve is $\beta_{1}$.  Label the squares such that the horizontal permutation of the origami starting from the leftmost square is $ \sigma_g $ and let $ \tau $ be the corresponding vertical permutation.

The map $ f $ lifts to a map $ \tilde{f} $ mapping $ R_1 $ to a square-tiling $ R_2 $ for $\left\{\alpha_{2}, \beta_{2} \right\}$.  This $ \tilde{f} $ also acts as a permutation of the labels of the squares.  Label the squares of $ R_2 $ in consecutive order from left to right and let $ k+1=\tilde{f}(1) $.  Since $ \tilde{f} $ is a homeomorphism of $\alpha_{1}$ onto $\alpha_{2}$, it must map consecutive squares to consecutive squares, and so $ \tilde{f}(i)=k+i $ for each $ i \in \lbrace 1, 2, \ldots, 2g-1 \rbrace $.  Now let $ \tau' $ be the vertical permutation of $\left\{\alpha_{2}, \beta_{2} \right\}$.  Since $ \tilde{f} $ is a homeomorphism of $\beta_{1} $ onto $ \beta_{2} $, the observation above implies that, if $ j=\tau(i) $, then $ j+k=\tau'(i+k) $ for any $ i,j \in \lbrace 1, 2, \ldots, 2g-1 \rbrace $.  This means that, if we express both permutations in cycle notation, then the expression for $ \tau' $ is obtained from the expression for $ \tau $ by adding $ k $ to each entry.  In other words: \[ \tau' = \sigma_{g}^{-k} \tau \sigma_{g}^{k} \ . \]So $ \tau $ and $ \tau' $ are conjugate by a permutation that commutes with $ \sigma_g $.

Conversely, suppose there exist square-tilings $ R_1 $ and $ R_2 $ for $\left\{\alpha_{1}, \beta_{1} \right\}$ and $\left\{\alpha_{2}, \beta_{2} \right\}$ (respectively) with vertical permutations $ \tau $ and  $ \tau' $ (respectively) and that $ \tau $ and $ \tau' $ are conjugate by a permutation that commutes with $ \sigma_g $ (i.e., an element of the centralizer of $ \sigma_g $).  This centralizer is simply the subgroup $\langle \sigma_g \rangle $ generated by $\sigma_{g}$, so the two vertical permutations must be conjugate by a power of $ \sigma_g $.  Then it's possible to define a map from $ R_1 $ to $ R_2 $ that maps consecutive squares to consecutive squares and, therefore, maps $ \alpha_1 $ onto $ \alpha_2 $ and $ \beta_1 $ onto $ \beta_2 $.  This map descends to a self-homeomorphism of $ S_g $ mapping $\left\{\alpha_{1}, \beta_{1} \right\}$ to $\left\{\alpha_{2}, \beta_{2} \right\}$, as claimed.

Conversely, note that any homeomorphism sending $\alpha_{i}$ to $\beta_{i}$ will lift to a square-preserving homeomorphism from $R_{1}$ to $R_{2}$, \textit{after} potentially replacing $R_{2}$ with an equivalent diagram $R'_{2}$ obtained by cyclically rotating the squares so that the first square of $R_{1}$ is sent to the first square of $R'_{2}$. This homeomorphism will necessarily preserve the combinatorics of the gluing pattern. It follows that the permutations $\tau_{\alpha_{1}, \beta_{1}}, \tau_{\alpha_{2}, \beta_{2}}$ must coincide, up to conjugation by the centralizer of $\sigma_{g}$. 

\end{proof}

\autoref{characterizing perm} implies that the problem of building minimal origami pairs can be replaced by the problem of finding $(2g-1)$-cycles $\tau$ whose commutator with $\sigma_{g} $ is a $(2g-1)$-cycle. We will call such permutations \textit{minimal origami permutations}. Since the centralizer of $\sigma_{g}$ has order $2g-1$, we have the following direct consequence of \autoref{characterizing perm}:

\begin{proposition} \label{counting} Let $O_{g}$ denote the number of $\mbox{Mod}^{\pm}(S_{g})$-orbits of minimal origami pairs and let $P_{g}$ be the number of distinct minimal origami permutations. Then 

\[ O_{g} \geq \frac{P_{g}}{2g-1} \ . \]

\end{proposition}

We conclude this section by again reframing the problem at hand in more convenient terms. Let $\rho$ be a $(2g-1)$-cycle so that: \medskip

\begin{enumerate}
\item $\rho \sigma_{g}^{-1}$ is a $(2g-1)$-cycle;
\item $\rho$ is conjugate to $\sigma_{g}$ by a $(2g-1)$-cycle. 
\end{enumerate} \medskip

Given such a $\rho$, let $\tau$ be the $(2g-1)$-cycle conjugating $\rho$ to $\sigma_{g}$. Thus 
\[ \sigma_{g} = \tau^{-1} \rho \tau \Rightarrow \rho = \tau \sigma_{g} \tau^{-1}. \]

But then
\[ [\tau, \sigma_g]=\tau \sigma_g \tau^{-1} \sigma_g^{-1}=\rho \sigma_g^{-1} \]
By the first property above, this commutator is a $(2g-1)$-cycle and by second property above, $\tau$ is a $(2g-1)$-cycle as well. Thus, by \autoref{characterizing perm}, $\tau$ is a minimal origami permutation. We replace the problem of finding the permutation $ \tau $ with the problem of finding the permutation $ \rho $.

\begin{lemma} \label{rephrase} Suppose $\rho_{1}, \rho_{2}$ satisfy the above two properties, and let $\tau_{1}, \tau_{2}$ be the $(2g-1)$-cycles that conjugate $\rho_{1}, \rho_{2}$ (respectively) to $\sigma_{g}$. Then if $\tau_1$ and $\tau_2$ are conjugate by an element of $\langle \sigma_{g} \rangle$, then the same holds for $\rho_1$ and $\rho_2$. 
\end{lemma}

\begin{proof} 




If there is some $k$ so that 
\[ \tau_2 = \sigma^{-k} \tau_1 \sigma^{k}, \]
then 
\[ \rho_2 = \tau_{2}\sigma_{g} \tau_{2}^{-1} \]
\[= (\sigma_{g}^{-k} \tau_{1} \sigma_{g}^{k}) \sigma_{g} (\sigma_{g}^{-k} \tau_{1}^{-1} \sigma_{g}^{k}) \]
\[ = \sigma_{g}^{-k}(\tau_{1}\sigma_{g} \tau_{1}^{-1})\sigma_{g}^{k} = \sigma_{g}^{-k}\rho_{1} \sigma_{g}^{k}.\]

\end{proof}

Thus, distinct solutions to the new permutation problem correspond to distinct solutions to the old one. We now go on to build solutions to the new problem.

\section{The main construction}
\label{main construction}

Our goal is to construct $(2g-1)$-cycles $\rho$ satisfying \medskip

\begin{enumerate}
\item $\rho \sigma_{g}^{-1}$ is a $(2g-1)$-cycle;
\item $\rho$ is conjugate to $\sigma_{g}$ by a $(2g-1)$-cycle
\end{enumerate} \medskip

We begin by outlining the construction in the case where $g$ is odd. Begin by placing the numbers $3,4,5,..., 2g-1,1,2$ (in that order) in a horizontal row. This will represent the permutation $\sigma_{g}$: placing parentheses around this list gives an expression of $\sigma_{g}$ in cycle notation. Placing the $1$ and $2$ at the end is a notational convenience whose purpose will be made clear later in the section. 

We will build our permutation $\rho$ by placing the numbers $\left\{1,2,..., 2g-1 \right\}$ in a second row above the $ \sigma_g $-row.  The permutation $\rho$ will be represented in cycle notation by placing parentheses around the numbers in this top row. The stacking of the two rows defines a third permutation $ \tau $ which maps each entry in the top row to the entry below it.  This is the permutation that conjugates $\rho$ to $\sigma_{g}$.

We proceed by first placing the number $ 1 $ in the top row above the number $ 3 $.  We then place each of the pairs $ (i, i+1) $ ($ i $ odd) in turn such that: \medskip

\begin{enumerate}
    \item The odd entry of the pair is placed over an even entry $ j $ of the bottom row (until there is no other choice but to place it over an odd entry, which will necessarily be $1$).
    \item The next pair to be placed in the top row is determined by the number $ j $.
\end{enumerate} \medskip

Since we placed $ 1 $ over $ 3 $, this determines that $ \tau(1)=3 $, so we begin by placing the pair $(3,2)$ over an arbitrary pair of numbers $(j, j+1)$ where $j$ is even. This dictates that $ \tau(1)=3 $, $\tau(3)= j$, and $\tau(2)= j+1$. Thus, the cycle expression for $\tau$ will contain a cycle that begins with $(1,3,j,...)$.  We do not yet know where the $2$ or the $j+1$ will appear in cycle notation, only that $2$ must appear directly before $j+1$ in some cycle.

Because $ 3 $ was placed over $ j $, the next pair to be placed in the top row is $ (j+1, j) $, which we place over some other pair of consecutive entries in the bottom row (making sure that $ j+1 $ is placed above an even number).  We continue in this manner until we have only one pair left to place (over the entries $ (1,2) $ in the bottom row; this was the reason for beginning the bottom row with $ 3 $ instead of $ 1 $).  We claim that this results in a cycle expression for $\tau$ which consists of a single $(2g-1)$-cycle, thus verifying the second of the two desired properties of $ \rho $.

\subsection{Example of the construction.}
\label{g=7 example}
We outline the above construction in an example for $g= 7$ (so $2g-1 = 13$). The construction begins with the following diagram: 

\[ \begin{array} {ccccccccccccc}
1 & . & . & . & . & . & . & . & . & . & . & . & . \\
3 & 4 & 5 & 6 & 7 & 8 & 9 & 10 & 11 & 12 & 13 & 1 & 2  
\end{array}\] \medskip

As described above, the first choice we make in the construction is to decide the placement of the pair $(3,2)$; we can choose arbitrarily between $(4,5), (6,7), (8,9), (10,11)$, and $(12,13)$ as numbers over which to place this pair. In this example, we will chose to place it over $(8,9)$, producing the following diagram: 

\[ \begin{array} {ccccccccccccc}
1 & . & . & . & . & 3 & 2 & . & . & . & . & . & . \\
3 & 4 & 5 & 6 & 7 & 8 & 9 & 10 & 11 & 12 & 13 & 1 & 2  
\end{array}\] \medskip

This dictates that $\tau(3)$ will be $8$, so the next choice will be the placement of $ (9,8) $; the possible options are $(4,5), (6,7), (10,11)$, and $(12,13)$. We will select $(12, 13)$, producing: 

\[ \begin{array} {ccccccccccccc}
1 & . & . & . & . & 3 & 2 & . & . & 9 & 8 & . & . \\
3 & 4 & 5 & 6 & 7 & 8 & 9 & 10 & 11 & 12 & 13 & 1 & 2  
\end{array}\] \medskip

We must next choose the position of the pair $(13,12)$, and we choose to place it over $(4,5)$: 

\[ \begin{array} {ccccccccccccc}
1 & 13 & 12 & . & . & 3 & 2 & . & . & 9 & 8 & . & . \\
3 & 4 & 5 & 6 & 7 & 8 & 9 & 10 & 11 & 12 & 13 & 1 & 2  
\end{array}\] \medskip

Now we can choose to place $ (5,4) $ over $(6,7)$ or $(10,11)$; we choose $(6,7)$: 

\[ \begin{array} {ccccccccccccc}
1 & 13 & 12 & 5 & 4 & 3 & 2 & . & . & 9 & 8 & . & . \\
3 & 4 & 5 & 6 & 7 & 8 & 9 & 10 & 11 & 12 & 13 & 1 & 2  
\end{array}\] \medskip

We now have only the pair $ (7,6) $ left to place and no choice but to place it over $(10,11)$, and finally to place $(11,10)$ over $(1,2)$: 

\[ \begin{array} {ccccccccccccc}
1 & 13 & 12 & 5 & 4 & 3 & 2 & 7 & 6 & 9 & 8 & 11 & 10 \\
3 & 4 & 5 & 6 & 7 & 8 & 9 & 10 & 11 & 12 & 13 & 1 & 2  
\end{array}\] \medskip

The top row is the cylce expression of our permutation $\rho=(1,13,12,5,4,3,2,7,6,9,8,11,10)$. We now verify the desired properties. The permutation $\tau$ is the map sending each number in the top row to the number directly below it in the bottom: 

\[ \tau= (1,3,8,13,4,7,10,2,9,12,5,6,11), \]
which is indeed a $13$-cycle, as desired. To check the first property, we verify that $\rho \sigma_{7}^{-1}$ is a $13$-cycle: 

\[ \rho \sigma_{7}^{-1} = (1,12,4,2,6,8,10,13,11,9,7,5,3). \]

\subsection{Proof of the construction in odd genus.}

\begin{theorem} \label{the construction} The permutation $ \rho $ produced by the construction described above always satisfies the desired two properties. 
\end{theorem} 

\begin{proof}[Proof. Odd genus] The reader is encouraged to refer back to the example above during the course of the argument.  To show that $ \tau $ is a $ (2g-1) $-cycle, we will argue that every entry of $ \tau $ is either in the cycle containing $ 1 $ (denoted $ \kappa_1 $) or in the cycle containing $ 2 $ (denoted $ \kappa_2 $); we conclude by showing that $ 2 \in \kappa_1 $ and, therefore, $ \kappa_1=\kappa_2 $, so the entire permutation consists of a single cycle.

The permutation is determined by the placement of $ (g-1) $ pairs of the form $ (i+1, i) $, where $ i $ is even.  The last such pair is placed over the entries $ 1 $ and $ 2 $ so that, if $ (k+1, k) $ is that pair, then $ 1 $ is the image of $ k+1 $ and $ 2 $ is the image of $ k $; thus $ k+1 \in \kappa_1 $ and $ k \in \kappa_2 $.  The entries of this pair, though, are the images of the entries of the $ (g-2)^{nd} $ pair to be placed; since $ k $ is the image of its odd entry, that odd entry belongs to $ \kappa_2 $ and, since $ k+1 $ is the image of the even entry, the even entry belongs to $ \kappa_1 $.

We can proceed in this manner, ``walking back" to the first pair to be placed; with each step, we alternate whether the odd entry or the even entry belongs to $ \kappa_1 $:  after an odd number of steps, the even entry belongs to $ \kappa_1 $ and, after an even number of steps, the odd entry belongs to $ \kappa_1 $.  Since $ g-1 $ is even, it will take an odd number of steps to reach the first pair, which means that the even entry of this pair belongs to $ \kappa_1 $.  But the first pair to be placed was $ (3,2) $, which implies that $ \kappa_1=\kappa_2 $ as we claimed.  Thus $ \tau $ is a $ (2g-1) $-cycle.

To verify that $\rho \sigma_{g}^{-1}$ is a $(2g-1)$-cycle, we first remark that, from the second entry onward, the entries of $\rho$ alternate in parity, starting with odd. Moreover, for $j$ even, the $j^{th}$ entry is always an odd number that is one larger than the $(j+1)^{st}$ entry. It follows that $\rho \sigma_{g}^{-1}$ has a cycle expression that starts with 

\[ (1, 3^{rd} \hspace{1 mm} \mbox{entry of} \hspace{1 mm} \rho, 5^{th} \hspace{1 mm} \mbox{entry of} \hspace{1 mm} \rho,..., (2g-1)^{st} \hspace{1 mm} \mbox{entry of} \hspace{1 mm} \rho...) \].

Note that this sequence contains all the even numbers in $ \lbrace 1, 2, \ldots, 2g-1 \rbrace $.  The $(2g-1)^{st}$ entry of $\rho$ is, of course, sent to $1$ by $\rho$, which is then sent to $2g-1$ by $\sigma_{g}^{-1}$. By construction, every odd entry $ i>1 $ of $ \rho $ is followed by the number $ i-1 $; thus $ \rho \sigma_g^{-1}(2g-1)=2g-3, \rho \sigma_g^{-1}(2g-3)=2g-5 $, and so on through all the odd numbers in $ \lbrace 1, 2, \ldots, 2g-1 \rbrace $ (ending with $ 3 $, which then maps back to $ 1 $).  Since it contains all possible entries, $ \rho \sigma_g^{-1} $ is a $ (2g-1) $-cycle, completing the proof.

\end{proof} 

\begin{corollary} \label{odd theorem} For odd $g$, 

\[ P_{g} \geq (g-2)! \]
\end{corollary}

\begin{proof} There are $g-2$ pairs of the form $(j,j+1)$ ($j$ even) in the bottom row.  For each such pair, we place above it a pair $(i+1, i)$ ($i$ even).  Thus there are $ g-2 $ places for the first $ i $-pair, $ g-3 $ for the second, and so on, until the last $ j $-pair has been taken and we are forced to place the last $ i $-pair above $ (1,2) $.  Thus, the number of $ \rho $'s that can be constructed is $ (g-2)! $.
\end{proof}

\begin{remark} \label{no conjugates} Lemma \ref{rephrase} implies that if solutions $\rho_1, \rho_2$ are not conjugate by an element of $\langle \sigma_{g} \rangle$, the corresponding minimal origami permutations will correspond to distinct $\mbox{Mod}^{\pm}(S)$ orbits of minimally intersecting filling pairs. Proposition \ref{counting} implies that at most $2g-1$ of the solutions constructed above can correspond to the same $\mbox{Mod}^{\pm}(S)$ orbit, but in fact, the orbits of our solutions under conjugation by $\langle \sigma_{g} \rangle$ are pairwise disjoint, as we now demonstrate. 

Fix a solution $\rho$; a conjugate of $\rho$ by an element of $\langle \sigma_{g} \rangle$ can be obtained by adding some $k$ (modulo $2g-1$) to each entry of the cycle notation for $\rho$. Assume that one of these conjugates produces another one of our solutions.

If $k$ is odd, the resulting permutation will not be built from pairs of the form $(i+1, i)$ with $i$ odd, as all of our solutions are, and thus we can assume that $k$ is even. Since $2g-1$ is odd, the unique solution to $k + x  = 1$ (mod $2g-1$) must be even-- denote this number by $j$. By construction, $j+1$ appears directly before $j$ in the cycle notation for $\rho$, and thus the conjugate $\sigma_{g}^{-k} \rho \sigma^{k}$ must have $2$ appearing directly before $1$. For this permutation to be one of our other solutions, $3$ must then occur immediately before $2$, as $(3,2)$ is one of the pairs from which our solutions are built. 

It follows that $(3,2)$ is the last pair placed in the construction of $\rho$; indeed, the last pair is always the one which cyclically appears directly in front of $1$. However, by design, $(3,2)$ is always the first pair placed. 

\end{remark}

\autoref{odd theorem} together with \autoref{no conjugates} immediately implies \autoref{main} for odd $g$. 

\subsection{The even case.} We next consider the case of even $g \geq 4$ (as shown in \cite{AH}, there are no filling pairs intersecting $3$ times on a genus $2$ surface). We begin with a permutation $ \xi $ on the numbers $ \lbrace 1, 2, \ldots, 2g-3 \rbrace $ obtained by the construction described above (since $ 2g-3=2(g-1)-1 $, this fits into the odd-genus case), and let $\eta$ denote the permutation conjugating $\xi$ to $\sigma_{g-1}$. Arbitrarily choose any odd $k \in \left\{3,5,7,..., 2g-3 \right\}$ not equal to $\eta^{-1}(1)$. In a horizontal row, we place the entries of $ \sigma_g $ in ascending order. As in the odd-genus case, we will build our permutation $ \rho $ by arranging the numbers $\left\{1,2,..., 2g-1 \right\}$ above this $ \sigma_g$-row.  The top row will then be a cycle expression for $ \rho $ and $ \tau $ will be the permutation that conjugates $ \rho $ to $ \sigma_g $.

We place each entry of $ \eta $ except for $ k $ above its image (so $ \tau(x)=\eta(x) $ for all $ x \notin \lbrace k, 2g-2, 2g-1 \rbrace $); we complete the top row by placing $k$ over $2g-1$, $2g-1$ over $2g-2$, and $2g-2$ over $\eta(k)$.

\subsection{Example in the even case} We give an example of the construction in the even case where $g= 10$, starting with the following $17$-cycle, which is a minimal origami permutation for $g=9$:

\[ \xi= (1,11,10,15,14,3,2,7,6,17,16,9,8,5,4,13,12). \] \medskip

The permutation $ \eta $ conjugating $ \xi $ to $ \sigma_{9} $ is 
\[ \eta= (1,3,8,15,6,11,4,17,12,2,9,14,7,10,5,16,13). \] \medskip

We arbitrarily choose the odd number $k=11$ and carry out the construction described above. We begin by placing the numbers from $1$ to $19$ in ascending order from left to right: 

\[ \begin{array} {ccccccccccccccccccc} 
1 & 2 & 3 & 4 & 5 & 6 & 7 & 8 & 9 & 10 & 11 & 12 & 13 & 14 & 15 & 16 & 17 & 18 & 19 
\end{array}\]

We then place each number other than $ k=11,\ 2g-2=18 $, and $ 2g-1=19 $ above its image under $ \eta $. This leads to the diagram 

\[ \begin{array} {ccccccccccccccccccc} 
13 & 12 & 1 & & 10 & 15 & 14 & 3 & 2 & 7 & 6 & 17 & 16 & 9 & 8 & 5 & 4 & &  \\
1 & 2 & 3 & 4 & 5 & 6 & 7 & 8 & 9 & 10 & 11 & 12 & 13 & 14 & 15 & 16 & 17 & 18 & 19 
\end{array}\] \medskip

Finally, we place $11$ over $19$, $19$ over $18$, and $18$ over $ \eta(11)= 4 $, yielding 

\[ \begin{array} {ccccccccccccccccccc} 
13 & 12 & 1 & 18 & 10 & 15 & 14 & 3 & 2 & 7 & 6 & 17 & 16 & 9 & 8 & 5 & 4 & 19  & 11 \\
1 & 2 & 3 & 4 & 5 & 6 & 7 & 8 & 9 & 10 & 11 & 12 & 13 & 14 & 15 & 16 & 17 & 18 & 19 
\end{array}\] \medskip

Denote by $\rho$ the permutation obtained by reading the top row from left to right and interpreting it as a $19$-cycle. We compute: 
\[ \rho \sigma_{10}^{-1} = (1,17,15,13,11,12,19,10,14,2,6,16,8,4,18,9,7,5,3), \] \medskip
verifying the first property of a minimal origami permutation. Letting $\tau$ denote the permutation conjugating $\rho$ to $\sigma_{10}$, we see that 
\[ \tau = (1,3,8,15,6,11,19,18,4,17,12,2,9,14,7,10,5,16,13), \]
verifying the second property.

To demonstrate the necessity of requiring that $k \neq \eta^{-1}(1)$, we include the following example shown to the authors by Luke Jeffreys. 

Beginning with $\xi = (1,3,2,5,4)$ in genus $3$, we can take $\eta$ to be $(1,3,4,2,5)$. Then suppose we choose $k= 5= \eta^{-1}(1)$ and proceed as above. This will yield the permutation $\rho= (1,3,2,7,5,6,4)$. One then computes that $\rho \sigma_{4}^{-1}$ is \textit{not} a $7$-cycle, as required:
\[ \rho \sigma_{4}^{-1} = (1,2,6,3)(4,7)(5) \]

\begin{proof}[Proof. Even genus]
We now prove that the construction demonstrated above always produces a minimal origami permutation. For $g$ even, let $\xi$ denote the $(2g-3)$-cycle used to construct the $(2g-1)$-cycle $\rho$ and note that, beginning with $ 1 $, the entries of $ \rho, \xi $ in cycle notation coincide \textit{until} the odd number $k$ appears in $ \xi $.  We will describe the cycle expression of $ \rho \sigma_g^{-1} $ one entry at a time; this description is diagrammed in \autoref{evencase} and we encourage the reader to refer to it as we go.

We begin the first (and, we claim, only) cycle of $ \rho \sigma_g^{-1} $ with 1 and assume, for the moment, that $ k \neq \xi(1) $.  To find the next entry, we evaluate $ \rho(1) $, which is equal to $ \xi(1) $, and is therefore an odd number greater than 1.  By construction, almost every odd entry of $ \xi $ is followed by its predecessor:  if $ i+1>1 $ is odd, then $ \xi(i+1)=i $.  This means that the $ 2^{nd} $ entry of $ \rho \sigma_g^{-1} $ will be the $ 3^{rd} $ entry of $ \xi $.  This entry will be even and the behavior described above will repeat: $ \rho \sigma_g^{-1} $ will land on each even entry of $\xi$ until we reach the position just before $ k $, since $k$ is to the right of $1$ in the top row owing to the fact that $\eta^{-1}(1) \neq k$. What we have just described accounts for the red portion of \autoref{evencase} at the top left.  Note that, if $ \xi(1)=k $, then what we describe next happens at the beginning of the cycle.

In the permutation $ \rho $, the position occupied by $ k $ in the permutation $ \xi $ is now occupied by the entry $ 2g-2 $, which is then sent to $ 2g-3 $ by $ \sigma_g^{-1} $.  This is the largest odd entry of $ \xi $, which means that the next entry of $ \xi $ (and also of $ \rho $) is its predecessor $ 2g-4 $, which $ \sigma_g^{-1} $ then sends to $ 2g-5 $.  From this point onward, the entries of $ \rho \sigma_g^{-1} $ will continue counting downwards through the odd entries of $ \xi $ (which agree again with the entries of $ \rho $) until we reach $ k $.  In the construction of $ \rho $, $ k $ was placed above $ 2g-1 $ as the last entry in the top row; this means that, if $ \rho(k)=j $, then $ j $ will be the last odd entry placed in $\xi $ (since $ j $ would have been placed above 1 in that construction and 1 has now been cycled to the front of the bottom row).  The permutation $ \sigma_g^{-1} $ sends $ j $ to $ j-1 $; but, since $ j $ is an odd entry of $ \xi $ greater than 1, $ j-1 $ is also the entry following $ j $ in the cycle notation of $ \xi $.  The entry $ j-1 $ was last in $ \xi $, which means that $ \rho(j-1)=1 $, so $ (\rho \sigma_g^{-1})(j-1)=2g-1 $. This paragraph has accounted for the black portion of \autoref{evencase} in the top right. 

The entry $ 2g-1 $ was placed above $ 2g-2 $ in the construction of $ \rho $ and $ k $ was placed above $ 2g-1 $, so $ \rho(2g-1)=k $, which $ \sigma_g^{-1} $ then sends to $ k-1 $.  But $ k-1 $ follows $ k $ in the cycle expression of $ \xi $, so now we return to the pattern at the beginning of the cycle:  the next entry of $ \rho $ is an odd entry of $ \xi $ greater than 1, which means that the next entry in both $ \rho $ and $ \xi $ is its predecessor. Thus the entries of $ \rho \sigma_g^{-1} $ now run through the even entries of $ \xi $ appearing after $k$ in $\xi$.  This continues until we reach the even entry $ \tau^{-1}(2g-3) $ (again, see \autoref{evencase}), which is just before 2g-1 near the end of the cycle.  From this position, $ \rho \sigma_g^{-1} $ sends us to $ 2g-2 $. The above accounts for the lower red portion of \autoref{evencase}.  

The entry $ 2g-2 $ in $ \rho $ occupies the position that $ k $ occupied in the cycle expression of $ \xi $, which means that the next entry in $ \rho $ will be $ k-1 $, which $ \sigma_g^{-1} $ sends to $ k-2 $.  And now we once again resume a previously-established pattern:  since $ k-2 $ is odd, the next entry of $ \rho $ is its predecessor $ k-3 $, which $ \sigma_g^{-1} $ sends to $ k-4 $.  The remaining entries of $ \rho \sigma_g^{-1} $ will be the odd numbers less than $ k $ in descending order, until we finally reach 1, completing the cycle.  The cycle we have just described consists of all the numbers in $ \lbrace 1, 2, \ldots, 2g-1 \rbrace $; since we have exhausted all possible entries of $ \rho \sigma_g^{-1} $, this permutation is a $ (2g-1) $-cycle, as claimed.  \medskip

\begin{figure}
\centering
\includegraphics[width= \textwidth, height= 80mm]{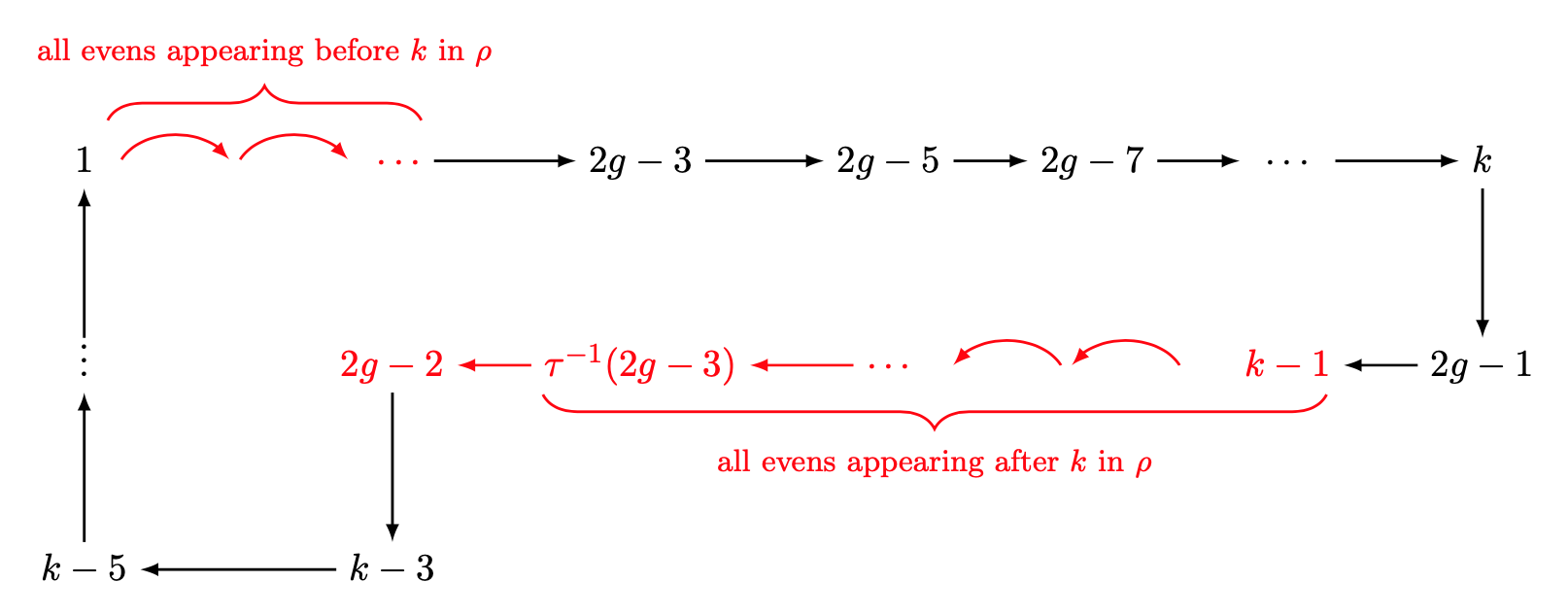}{}
\caption{}
\label{evencase}
\end{figure}

It remains to show that $\tau$, a permutation conjugating $\rho$ to $\sigma_{g}$, is a $(2g-1)$-cycle. A representative for the permutation $\tau$ in cycle notation beginning with the cycle containing $1$ will agree with the representative for $\eta$ that begins with $1$, until the occurrence of $k$ within $\eta$. At that point, the next entries of $\tau$ will be $2g-1, 2g-2, \eta(k)$, and after this, the subsequent entries will coincide exactly with what occurs in $\eta$ after $\eta(k)$. It follows that $\tau$ is a $(2g-1)$-cycle. \end{proof}

\vspace{3 mm} 

 Since there are $g-3$ odd numbers between $3$ and $2g-3$ (inclusive) not equal to $\eta^{-1}(1)$, we obtain: 

\begin{corollary}  For $g$ even, 
\[ P_{g} \geq (g-3)! \cdot (g-3) = (g-3)(g-3)! \]
\end{corollary}

\begin{remark} \label{no conjugates even} Analogous to Remark \ref{no conjugates}, we now show that the solutions described above can be conjugate to one another by elements of $\langle \sigma_{g} \rangle$ only under very specific circumstances. To this end, assume there are solutions $\rho_1, \rho_2$ constructed above and some $k$ so that $\sigma_{g}^{-k} \rho_1 \sigma_{g}^{k} = \rho_2$. 

Let $t_{1}$ (resp. $t_{2}$) denote the odd number in $\left\{3,..., 2g-3 \right\}$ used to construct $\rho_{1}$ (resp. $\rho_{2}$) from solutions in genus $g-1$. In cycle notation for $\rho_{i}$, each odd number $n$ except for $1$ and $2g-1$ appears directly before $n-1$. Moreover, $2g-1$ appears directly before $t_{i}$ and $2g-2$ appears directly before $t_{i}-1$. It follows that 
\[ \rho_{i}(2g-1)- 1 = \rho_{i}(2g-2). \]

Adding $k$ (modulo $2g-1$) to each entry of $\rho_{1}$ produces $\rho_{2}$. Let $j \in \left\{1,..., 2g-1 \right\}$ satisfy $j+k = 2g-1$; since $k$ is (wlog) not a multiple of $2g-1$, $j \neq 2g-1$. 

If $j=1$, we can take $k= 2g-2$ and thus $2g-1 + k = 2g-2$ (mod $2g-1$), so the position of $2g-2$ in $\rho_{2}$ will coincide with the  position of $2g-1$ in $\rho_{1}$. Hence, $t_{2} - 1 = t_{1} + k$ (mod $2g-1$), which implies that $t_{1} = t_{2}$. The position of $t_{2}$ in $\rho_{2}$ coincides with the position of $\rho_{1}(1)$ in $\rho_{1}$, and thus $\rho_{1}(1)= t_{2}+1= t_{1}+1$, which is even. By construction, this is only possible if $\rho_{1}(1) = 2g-2$, and this would imply that $t_{1} = 2g-3$. 

Next, assume that $j= t_{1}$. In this case we can take $k= 2g-1 -t_{1}$ (which is even), and this implies that the position of $2g-1$ in $\rho_{1}$ coincides with the position of $2g-1 - t_{1}$ in $\rho_{2}$. Any even number $n$ in $\rho_{2}$ either has the property that it is directly preceded by $n+1$, or it is one of $t_{2}-1, 2g-2$. If $2g-1 - t_{1}$ is preceded by $2g-t_{1}$ in $\rho_2$, then $1$ directly precedes $2g-1$ in $\rho_{1}$, contradicting the construction of $\rho_{1}$. 

If $2g-1-t_{1} = 2g-2$, then $\rho_{2}(2g-1-t_{1})= t_{2}-1$, but by assumption this is also $t_{1}+k$ since the position of $2g-1$ in $\rho_1$ coincides with the position of $2g-1-t_{1}$ in $\rho_{2}$. Thus, (mod $2g-1$) we obtain 
\[ t_{2}- 1 = t_{1}+k = t_{1} + 2g-1 - t_{1} = 2g-1, \]
which implies that $t_{2} = 1$, a contradiction. 

If $2g-1 -t_{1} = t_{2}- 1$, then $2g-2$ must immediately precede it in $\rho_{2}$. It then follows that $t_{1}- 1$ immediately precedes $2g-1$ in $\rho_{1}$. This in turn implies that at the end of the cycle for $\rho_1$, one sees $(...2g-2, t_{1}-1, 2g-1, t_{1})$. This specifies $\rho_1$ amongst all solutions obtained from the same permutation in odd genus: $\rho_1$ is the one for which the associated odd number $t_{1}$ is $\eta^{-1}(2g-3)$. 

Next, assume $j$ is odd and not equal to either $1$ or $t_{1}$. Then $\rho_{1}(j)= j-1$, and so $j-1+k = t_{2}$. However, $k$ is even in this case because both $2g-1$ and $j$ are odd, so $j-1+k$ is even, but $t_{2}$ is odd. 

On the other hand, if $j$ is even, $k$ is odd, and this will imply that one sees $g-3$ occurrences of $(... i, i-1,...)$ with $i$ even in the cycle for $\rho_{2}$, which is impossible by construction. 

To summarize: given $\rho_1$, there exists $\rho_2$ conjugate to it by an element of $\langle \sigma_{g} \rangle$ only if it is constructed from a genus $g-1$ solution by choosing either $\eta^{-1}(2g-3)$ or $2g-3$ as the odd number centered in the construction of an even genus solution from an odd one. 

Thus, given a solution in odd genus, there are $g-5$ ways of extending it to a solution in even genus such that none of the resulting permutations are conjugate by $\langle \sigma_{g} \rangle$.

\end{remark}

\section{Exploring $SL(2, \mathbb{Z})$-orbits} 

As discussed in the Introduction, the $SL(2, \mathbb{R})$-action on $\mathcal{H}(2g-2)$ restricts to an $SL(2, \mathbb{Z})$-action on the origamis contained in the minimal stratum.

The $SL(2, \mathbb{Z})$-action can be easily described in terms of the generators 

\[ A = \left( \begin{array} {cc} 
			1 & 1 \\
			1 & 0 
			\end{array} \right),  B = \left( \begin{array} {cc} 
									0 & -1 \\
									1 & 0 
									\end{array} \right).  \]
The generator $A$ has the effect of a rightward shear, transforming the unit square into a parallelogram by rotating its vertical edges $45$ degrees clockwise; the generator $B$ rotates the plane by $90$ degrees counter-clockwise. We define the action of either $A$ or $B$ on an origami by cutting the origami into its constituent squares, applying the appropriate generator to each square individually (after identifying it with the standard unit square in $\mathbb{R}^{2}$), and then gluing the images back together. The $B$ generator thus sends origamis to origamis. The shear generator $A$ yields a surface tiled by congruent parallelograms, each comprised of two isosceles right triangles; by cutting and re-gluing right triangles, one obtains a new square-tiling of the surface. 

Since there are only finitely many ways to glue $2g-1$ squares together, the orbit of any origami constructed in Section $4$ is finite. The cardinality of the orbit can be related to the index $ |SL(2, \mathbb{Z}) : V| $, where $ V $ is the Veech group of the origami. 

The number of squares is an obvious invariant of the $SL(2, \mathbb{Z})$-action. A slightly subtler invariant is the \textit{monodromy}, which is the (conjugacy class of the) subgroup of the symmetric group $ \Sigma_N $ generated by the vertical and horizontal permutations (where $N$ is the number of squares).

In collaboration with Zichen Cui, Ajeet Gary, Paige Helms, Ionnis Iakovoglou, Tasha Kim, and J.T. Rustad, we implemented the construction outlined in the previous section in order to form a computational library of origamis in $\mathcal{H}(2g-2)$ with $2g-1$ squares and one cylinder in each of the vertical and horizontal cylinder decompositions, for $g=3,4,5,6,7$. We summarize some of the findings in \autoref{table}.

\begin{figure}[htp]
\centering
\includegraphics[width= .8\textwidth, height= 80mm]{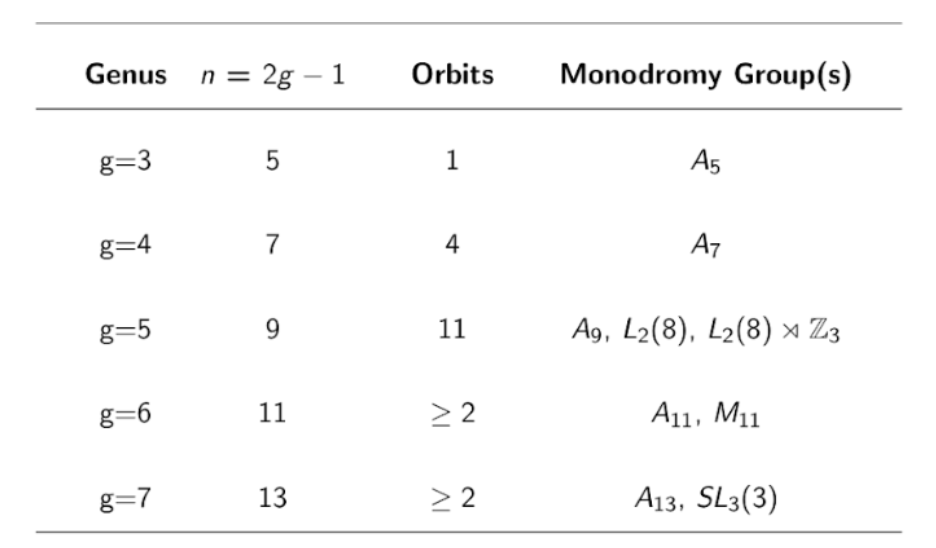}{}
\caption{}
\label{table}
\end{figure}

In fact, for $g \leq 5$, the search spaces are sufficiently small that it is possible to find all minimal origami permutations. Thus, for these genera, we determine the exact number of orbits of origamis in $\mathcal{H}(2g-2)$ with $2g-1$ squares and one cylinder for each of the vertical and horizontal cylinder decompositions. 

The group $M_{11}$ is the sporadic Mathieu group on $11$ symbols, while $L_{2}(8)$ is the group that admits the following presentation: 
\[ \langle a,b \ | \  a^{2} = b^{3}= ((ababab^{-1})^{2}ab^{-1})^{2}= 1 \rangle \]
While we conjecture that the number of orbits grows super-exponentially as a function of $g$ (and, furthermore, that many orbits can be found amongst the origamis constructed in Section $4$), we note that we do not even know if the number of orbits is monotonically increasing in $g$.  

\bibliographystyle{abbrv}
\bibliography{AougabMenascoNieland}

\end{document}